\documentclass[10pt]{amsart}
\usepackage{amsmath,amssymb,amsfonts}
\usepackage{enumerate}
\input xy
\xyoption{all}

\theoremstyle{plain}
\newtheorem{thm}{Theorem}
\newtheorem{theorem}{Theorem}[section]

\newtheorem{lemma}{Lemma}

\theoremstyle{definition} \theoremstyle{definition}

\newtheorem{rem}[theorem]{Remark}           

\theoremstyle{remark}

\newcommand{\A}{\mathbb{A}}

\newcommand{\C}{\mathbb{C}}

\newcommand{\Hom}{{\rm Hom}\,}

\def\M{{\rm M}}

\def\GSp{{\rm GSp}}
\def\PGSp{{\rm PGSp}}
\def\Sp{{\rm Sp}}

\def\GO{{\rm GO}}
\def\GL{{\rm GL}}

\def\GSO{{\rm GSO}}

\def\SO{{\rm SO}}
\def\O{{\rm O}}

\def\tr{{\rm tr\,}}

\begin{document}

\title{A remark on the failure of multiplicity one for $\GSp(4)$}
\author{Daniel File and Ramin Takloo-Bighash}
\vspace{2mm}

\dedicatory{To the memory of Ilya Piatetski-Shapiro}
\address{D. F.: Departement of Mathematics, 14 MacLean Hall, Iowa City, Iowa 52242-1419}
\email{daniel-file@uiowa.edu}
\address{R. T.-B.: Department of Mathematics, Statistics, and Computer Science, University of Illinois at Chicago,
851 S. Morgan St, Chicago, IL 60607} \email{rtakloo@math.uic.edu}
\thanks{
The second author's research is partially supported by the NSF,
NSA, and a research grant from the University of Illinois at
Chicago. We wish to thank Dipendra Prasad for
extremely useful communications}  \maketitle
\date{}
\begin{abstract}
We revisit a classical result of Howe and Pitatski-Shapiro on the failure of strong multiplicity one for $\GSp(4)$.
\end{abstract}

\section{Introduction}

Let $F$ be a number field. The strong multiplicity one theorem for
cuspidal automorphic representations of $\GL(n)$ due to
Piatetski-Shapiro \cite{Pia} states that if  $\pi_1=\otimes_v
\pi_{1,v}$ and $\pi_2=\otimes_v \pi_{2,v}$ are cuspidal
automorphic representations of $\GL(n)$ such that for all but finitely many $v$ we have $\pi_{1,v} \cong
\pi_{2,v}$, then $\pi_1 = \pi_2$. In his proof Piatetski-Shapiro
used the uniqueness of the Whittaker model. The $n=2$ case of this
theorem was proved by Casselman \cite{C}.

The strong multiplicity one theorem does not hold for other
classical groups,  such as symplectic groups. When two
representation $\pi_1$ and $\pi_2$ have the property that
$\pi_{1,v} \cong \pi_{2,v}$ for all $v$ outside of a finite set of
places of $F$, they are said to be {\em nearly equivalent}. Howe
and Piatetski-Shapiro \cite{H-PS} constructed examples of nearly
equivalent representations for $\Sp(4)$ that are not isomorphic.
Cogdell and Piatetski-Shapiro showed that for any positive integer
$n$, there exist inequivalent cuspidal automorphic representations
$\pi_1$, \dots ,$\pi_n$ of $\PGSp(4)$ that are nearly equivalent
\cite{C-PS}. However, if $\pi_1$ and $\pi_2$ are nearly equivalent
generic automorphic representations for $\GSp(4)$, then Soudry
\cite{So} showed that $\pi_1=\pi_2$.

Let $K/F$ be a quadratic extension, and let $T$ be a torus in
$\GL(2)$ $F$-isomorphic to $K^\times$. Let $\chi$ be a
grossencharacter of $K$. The purpose of this note is to prove the
following theorem:

\begin{thm}\label{thm1}
Let $F$ be a totally real number field, and let $S$ be a non-empty
finite set of inert primes of $K/F$ of even cardinality. Then there are
two automorphic cuspidal representations of $\GSp(4)$ $\pi=\otimes_v' \pi_v,
\pi' = \otimes_v' \pi_v'$ such that
\begin{itemize}
\item $\pi$ is generic, but has no $(T, \chi)$-Bessel model;
\item $\pi'$ is not generic, but has a $(T, \chi)$-Bessel model;
\item for all $v \not\in S$, $\pi_v \simeq \pi_v'$;
\item for all $v \in S$, $\pi_v \not\simeq \pi_v'$.
\end{itemize}
\end{thm}

The restriction on the set $S$ is to guarantee the existence of a non-split quaternion algebra $D$ over $F$ containing $K$
ramified precisely at the places in $S$. We fix this quaternion algebra throughout the paper.

\

As in the classical paper of \cite{H-PS} the proof of this theorem
uses theta correspondence from two different orthogonal groups.
Here, however, Bessel coefficients, locally and globally, play a
prominent role. One of the main difficulties in \cite{H-PS} is constructing
representations on a non-split orthogonal group whose theta lift
to $\GSp(4)$ is non-zero. We use a globalization theorem of Prasad
and Schulze-Pillot to construct certain cuspidal representations
on the non-split orthogonal group. Using Bessel coefficients we prove
that these representations have non-zero theta lift
to $\GSp(4)$.

\

That the strong multiplicity one theorem for $\GSp(4)$ fails is of
course well-known, see e.g. \cite{H-PS}, \cite{Ro}. The
contribution of this modest note, if any, is to show how easily
multiplicity one fails and how prevalent this phenomena is.

\

A bit of notation. In this paper $\GSp(4)$ is the group of
similitude transformations of a four dimensional symplectic space.
If $D$ is a quaternion algebra, the Jacquet-Langlands transfer of
a representation $\pi$ of $D^\times$, locally and globally, is
denoted by $\pi^{JL}$. If $\pi$ is a representation of $\GL(2)$,
again locally and globally, its Jacquet-Langlands transfer to
$D^\times$ is denoted by $\pi_{JL}$. If $\pi$ is not square
integrable, we define $\pi_{JL}=\{0\}$.

\

This paper is organized as follows. Section \ref{sect:background}
contains preliminaries. The proof of the theorem is presented in
Section \ref{sect:proof}.

\section{Background and preliminaries}\label{sect:background}
\subsection{Orthogonal groups}
Let $V$ be the vector space $\M_2$, of the two by two matrices,
equipped with the quadratic form $\det$. Let $(,)$ be the
associated non-degenerate inner product, and $H = \GO(V, (,))$ be
the group of orthogonal similitudes of $V$, $(,)$. The group
$\GL(2) \times \GL(2)$ has a natural involution $t$ defined by
$t(g_1, g_2) = (^t g_2 ^{-1}, ^t g_1^{-1})$, where the superscript
$t$ stands for the transposition. Let $\tilde{H} = (\GL(2) \times
\GL(2)) \rtimes <t>$ be the semi-direct product of $\GL(2) \times
\GL(2)$ with the group of order two generated by $t$. There is a
sequence
\begin{equation}
1 \longrightarrow \mathbb{G}_m \longrightarrow \tilde{H}
\longrightarrow H \longrightarrow 1,
\end{equation}
where the homomorphism $\rho: \tilde{H} \rightarrow H$ is defined
by $\rho(g_1, g_2)(v) = g_1 v g_2^{-1}$, and $\rho(t) v = \,^t v$,
for all $g_1, g_2 \in \GL(2)$ and $v \in V$. Also, $\mathbb{G}_m
\rightarrow \tilde{H}$ is the natural map $z \mapsto (z,z)\times 1
$. It follows that the image of the subgroup $\GL(2) \times \GL(2)
\subset \tilde{H}$ under $\rho$ is the connected component of the
identity of $H$. We usually denote $H$ by $\GO(2,2)$.

Similarly, if $D$ is a non-split quaternion algebra we may repeat the above discussion
 with the reduced norm instead of $\det$. The corresponding general orthogonal group is denoted by $\GO(4)$.

\

Let $X$ be four dimensional and either split or anisotropic. Then
a representation $\pi$ of $\GSO(X)$ over a local field is given by
a pair of representations $\pi_1, \pi_2$ of either $\GL(2)$ or
$D^\times$ with the same central character. If $\pi_1 \not\simeq
\pi_2$, the $\pi$ is called \emph{regular}; otherwise, we call it
\emph{invariant}. If $\pi$ is regular, then
$$
ind_{\GSO(X)}^{\GO(X)} \pi
$$
is irreducible and is denote by $(\pi_1, \pi_2)^+$. We also
formally set $(\pi_1, \pi_2)^-=\{0\}$. In contrast, if $\pi$ is
invariant, this induction is reducible and has two irreducible
pieces $(\pi_1^+, \pi_2)$, $(\pi_1, \pi_2)^-$. There is a
canonical characterization of the representations $(\pi_1,
\pi_2)^\pm$. Pick an anisotropic vector $y \in X$ and let $Y$ be
its orthogonal complement. The stabilizer of $y$ in $\O(X)$ is
$\O(Y)$. We say $\pi$ is distinguished if
$$
\Hom_{\SO(Y)}(\pi, 1) \ne 0.
$$
If $\pi$ is distinguished then the above $\Hom$ space is in fact
one dimensional. Then
$$
\dim_\C \Hom_{\O(Y)}(\pi^+, 1) + \dim_\C \Hom_{\O(Y)}(\pi^-, 1) =
1.
$$
By definition $\pi^+$ is the representation that affords the
non-zero functional. Equivalent characterizations of $\pi^\pm$
appear in \cite{H-PS}. Then in the cases of our interest an
invariant representation is distinguished.

\

There is an explicit transfer of automorphic representations, a
Jacquet-Langlands transfer (``{\em JL}") from $\GO(4)$ to
$\GO(2,2)$. We follow \S 7 of \cite{H-PS} where the map is defined
to be
$$
(\pi_1, \pi_2)^{\pm} \mapsto (\pi_1^{JL}, \pi_2^{JL})^{\pm},
$$
locally. The global map is constructed by patching local maps
together.

\subsection{Theta correspondence} Our reference for theta correspondence for dual pairs $(\GO(X), \GSp(W))$
is \cite{Harris-Kudla}. For the particular case of our interest
where $\dim X = \dim W =4$ the local non-archimedean theory has
been worked out in \cite{Roberts}. We give a brief summary of this
theory.

The main theorem of \cite{Roberts} asserts:
\begin{thm}\label{thm:distinguished}
Let $\sigma$ be an irreducible admissible representation of
$\GO(X)$ with $X$ a four dimensional quadratic space over a
non-archimedean local field. Then $\sigma$ has non-trivial theta
lift to $\GSp(4)$ if and only if $\sigma$ is not of the form
$\pi^-$ for some distinguished irreducible admissible
representation $\pi$ of $\GSO(X)$.
\end{thm}
\begin{proof}
  Theorem 6.8 of \cite{Roberts}.
\end{proof}

\

\ In the global setting the theta lift of a generic representation
of $\GO(2,2)$ to $\GSp(4)$ is non-zero if there is no local
obstruction. More generally we have the following theorem:

\begin{thm}\label{thm:takeda}
Let $\pi_1, \pi_2$ be a pair of non-one dimensional automorphic
cuspidal representations of $D^\times$ or $\GL(2)$ over a totally
real field $F$. Let $\pi$ be an automorphic cuspidal
representation of $\GO(2,2)$ or $\GO(4)$ whose restriction to the
corresponding $\GSO$ contains $\pi_1 \otimes \pi_2$. Then the
theta lift of $\pi$ to $\GSp(4)$ is non-zero if there is no local
obstruction.
\end{thm}
\begin{proof}
See Theorem 1.3. of \cite{Takeda}.
\end{proof}

\subsection{Bessel models} Let $S \in M_2(F)$ be a symmetric matrix. We define a subgroup $T$ of $\GL(2)$ by
\begin{equation*}
T = \{ g \in \GL(2)\, \vert \, \,^t g S g = \det g. S \}.
\end{equation*}
Then we consider $T$ as a subgroup of $\GSp(4)$ via
\begin{equation*}
t \mapsto \begin{pmatrix} t \\ & \det t. \,^t t ^{-1}
\end{pmatrix},
\end{equation*}
$t \in T$.

We denote by $U$ the subgroup of $\GSp(4)$ defined by
\begin{equation*}
U = \{ u(X) = \begin{pmatrix} I_2 & X \\ & I_2 \end{pmatrix} \,
\vert \, X = \,^t X \}.
\end{equation*}
Finally, we define a subgroup $R$ of $\GSp(4)$ by $R = TU$.

Let $\psi$ be a non-trivial character of $F \backslash \A$, and define a character $\psi_S$ on $U(\A)$ by $\psi_S(u(X)) =
\psi(\tr (SX))$ for $X = \,^t X \in \M_2(\A)$. Let $\Lambda$ be a character of $T(F) \backslash T(\A)$. Denote by
$\Lambda \otimes \psi_S$ the character of $R(\A)$ defined by
$(\Lambda \otimes \psi)(tu) = \Lambda(t) \psi_S(u)$ for $t \in
T(\A)$ and $u \in U(\A)$.

Let $\pi$ be an automorphic cuspidal representation of
$\GSp_4(\A)$ and $V_\pi$ its space of automorphic functions. We
assume that
\begin{equation}\label{compatible}
\Lambda \vert_{\A^\times} = \omega_\pi.
\end{equation}
Then for $\varphi \in V_\pi$, we define a function $B_\varphi$ on
$\GSp_4(\A)$ by
\begin{equation}\label{Bessel}
B_\varphi(g) = \int_{Z_\A R_F\backslash R_\A} (\Lambda \otimes
\psi_S)(r)^{-1}. \varphi(rh) \, dh.
\end{equation}
We say that $\pi$ has a global Bessel model of type $(T, \Lambda)$ for $\pi$ if for some $\varphi \in V_\pi$, the function
$B_\varphi$ is non-zero. In this case, the $\C$-vector space of
functions on $\GSp_4(\A)$ spanned by $\{ B_\varphi \, \vert \,
\varphi \in V_\pi \}$ is called the space of the global Bessel
model of $\pi$.

Similarly, one can consider local Bessel models. Fix a local field
$F$. Define the algebraic groups $T_S$, $U$, and $R$ as above.
Also, consider the characters $\Lambda$, $\psi$, $\psi_S$, and
$\Lambda \otimes \psi_S$ of the corresponding local groups. Let
$(\pi, V_\pi)$ be an irreducible admissible representation of the
group $\GSp(4)$ over $F$. Then we say that the representation
$\pi$ has a local Bessel model of type $(T, \Lambda)$ if
there is a functional $\lambda_B \in (V_\pi^\infty)'$, a
linear functional on $V_\pi^\infty$ satisfying
\begin{equation*}
\lambda_B(\pi(r) v ) = (\Lambda \otimes \psi_S)(r)\lambda_B(v),
\end{equation*}
for all $r \in R(F)$, $v \in V_\pi$.

The fundamental properties of Bessel models in the local setting are established in \cite{P-TB}.

\subsection{Tunnell dichotomy theorem}
The following theorem is fundamental in this work:

\begin{thm}[Tunnell Dichotomy Theorem]\label{thm:tunnell} Let $F$ be local field. Let $T$ be a torus in $\GL(2)$, and $\lambda$ a quasi-character of $T(F)$. For every
irreducible admissible representation $\Pi$ of $\GL_2(F)$, we
have
\begin{equation}
\dim_\C \Hom_{T(F)}(\Pi, \lambda) + \dim_\C \Hom_{T(F)}(\Pi_{JL},
\lambda) =1.
\end{equation}
Here $\Pi_{JL}$ is the Jacquet-Langlands lift of $\Pi$ to the
unique quaternion algebra $D$ over $F$.  Hence, $\dim_\C
\Hom_{T(F)}(\Pi^{JL}, \lambda)=0$ if $\Pi$ is not discrete series,
or if $T(F)$ is split.
\end{thm}
\begin{proof}
  See \cite{saito} and \cite{tunnell}.
\end{proof}

If for some representation $\Pi$, the space
$$
\Hom_{T(F)}(\Pi, \lambda) \ne 0
$$
we say that $\Pi$ has a $(T, \chi)$-Waldspurger model. There is a similar notion in the global setting.

\subsection{A globalization theorem}

We also recall the following globalization theorem of Prasad and Schulze-Pillot \cite{P-S}:
\begin{thm}[Globalization Theorem]\label{thm:global}
Let $H$ be a closed algebraic subgroup of a reductive group $G$ both defined over a number field $F$. Let $Z$ be the identity component of the center of $G$. Assume $Z \leq H$ and that $H/Z$ has no $F$-rational characters. Let $\chi = \otimes_v' \chi_v$ be a one dimensional automorphic representation of $H(\A)$. Suppose $S$ is a finite set of non-archimedean places of $F$ and for each $v \in S$ we are given an irreducible supercuspidal representation $\pi_v$ such that
$$
\Hom_{H(F_v)}(\pi_v, \chi_v) \ne 0.
$$
Let $S'$ be a finite set of places
containing $S$ and all the infinite places, such that $G$ is quasi-split at places outside $S'$, and $\chi_v$ is unramified outside $S'$.
Then there exists an automorphic cuspidal representation $\Pi = \otimes_v' \Pi_v$
of $G(\A)$ such that
\begin{itemize}
\item $\Pi_v = \pi_v$ for $v \in S$;
\item $\Pi_v$ is unramified at all finite places of $F$ outside $S'$; and
\item there is an $f \in \Pi$ such that
$$
\int_{H(F) Z(\A) \backslash H(\A)} f(h) \chi(h)^{-1} \, dh \ne 0.
$$
\end{itemize}
\end{thm}
\begin{proof}
  See Theorem 4.1 of \cite{P-S}.
\end{proof}
We will apply this theorem in the following manner. Let $T$ be a torus embedded in the multiplicative group of a non-split quaternion algebra $D$ defined over a number field $F$.
Let $\chi = \otimes_v' \chi_v$ be an one dimensional automorphic representation of $T$. Since $T$ is anisotropic, $T/Z$ does not have any $F$-rational characters.
Let $v$ be a place where $T$ is split, and let $\pi_v$ be a supercuspidal representation of $D^\times(F_v) \simeq \GL_2(F_v)$. Then it follows from Theorem \ref{thm:tunnell} that
$$
\Hom_{T(F_v)}(\pi_v, \chi_v) \ne 0.
$$
Let $S'$ be the finite set of places consisting of places in $S$, archimedean places, places where $\chi$ ramifies, and $v$. Then the theorem asserts that there is an automorphic cuspidal representation
$\Pi = \otimes_v' \Pi_v$ of $D^\times(\A)$
unramified outside $S'$ such that $\Pi_v \simeq \pi_v$ and $\Pi$ has a global $(T, \chi)$-Waldspurger model.

\section{Proof of the theorem}\label{sect:proof}

As in \cite{H-PS} we will
use the following diagram
$$
\xymatrix{
\GO(2,2) \ar[dr]^\theta \\
& \GSp(4) \\
\GO(4) \ar[ur]^\theta \ar@{.>}[uu]^{JL}
}
$$
Here $\GO(4)$ is constructed using a non-split quaternion algebra
$D$, ramified precisely at the places in $S$, that contains the
torus $T$. Such a quaternion algebra exists if $K/F$ is an
extension of fields.  One can then consider in the local situation
the theta lift of the representations $(\pi_1, \pi_2)^+$ and
$(\pi_1^{JL}, \pi_2^{JL})^+$ to $\GSp(4)$; we denote these
representations by $\theta(\pi_1, \pi_2)$ and $\theta(\pi_1^{JL},
\pi_2^{JL})$, respectively.
\begin{lemma}\label{lemma1}
In the local situation, the representations $\theta(\pi_1, \pi_2)$
and $\theta(\pi_1^{JL}, \pi_2^{JL})$ are non-zero and are
inequivalent.
\end{lemma}
\begin{proof}
In the archimedean situation this lemma is well-known. The
non-vanishing of the local theta lifts follows from Remark 6.8 of
\cite{P-TB} as follows. We will show that $\theta(\pi_1, \pi_2)$
and $\theta(\pi_1^{JL}, \pi_2^{JL})$ have some non-zero Bessel
model. To see this it suffices to show that $\pi_1, \pi_2$, and
$\pi_1^{JL}, \pi_2^{JL}$ respectively, have a common Waldspurger
model. For $\GL(2)$ representations this is obvious, as by Theorem
\ref{thm:tunnell} every representation has almost all Waldspurger
models for a given torus. For $D^\times$, if $\pi_1 \simeq \pi_2$,
this is obvious; if $\pi_1 \not\simeq \pi_2$, we need to show the
existence of a torus whose trivial characters occurs in $\pi_1
\otimes \tilde{\pi}_2$. For this see \cite{Prasad2}. The
non-vanishing of the theta lift also follows from Theorem
\ref{thm:distinguished}.

In order to prove that the two representations are inequivalent we
will use Bessel models. Namely we will show the existence of a
Bessel model for the representation $\theta(\pi_1^{JL},
\pi_2^{JL})$ that other representations cannot have. By Corollary
7.1 of \cite{P-TB} the representation $\theta(\pi_1^{JL},
\pi_2^{JL})$ will have a $(T, \chi)$ Bessel model if and only if
the two representations $\pi_1^{JL}$, $\pi_2^{JL}$ have $(T,
\chi)$-Waldspurger models. It follows from Theorem
\ref{thm:tunnell} that if $T$ is split, $\pi_1^{JL}$ and
$\pi_2^{JL}$ will have $(T, \chi)$-Waldspurger models. On the
other hand, by the same corollary since $\pi_1$, $\pi_2$ do not
have $(T, \chi)$-Waldspurger models, $\theta(\pi_1, \pi_2)$ cannot
have $(T, \chi)$-Bessel models.
\end{proof}

\begin{lemma}\label{lemma2}
Let $D$ be a non-split quaternion algebra over a totally real
number field $F$ which is ramified at places in a set $S$. Let
$(\pi_1, \pi_2)$ be a pair of non-isomorphic automorphic
representations of $D^\times(\A_F)$, and suppose that
$\theta(\pi_1, \pi_2)$ and $\theta(\pi_1^{JL}, \pi_2^{JL})$ are
non-zero. Then
\begin{itemize}
  \item for $v \in S$, $\theta(\pi_1, \pi_2)_v \not\simeq \theta(\pi_1^{JL}, \pi_2^{JL})_v$;
  \item for $v \not\in S$, $\theta(\pi_1, \pi_2)_v \simeq \theta(\pi_1^{JL}, \pi_2^{JL})_v$.
\end{itemize}
\end{lemma}
\begin{proof}
  This is a direct consequence of the previous lemma.
\end{proof}

We can now present the proof of the main theorem.

\begin{proof}[Proof of Theorem \ref{thm1}]
By Lemma \ref{lemma2} it suffices to find a pair of non-equivalent
automorphic representations $(\pi_1, \pi_2)$ of $D^\times(\A_F)$
such that $\theta(\pi_1, \pi_2)$ and $\theta(\pi_1^{JL},
\pi_2^{JL})$ are non-zero. In order to do this we will use the
discussion following Theorem \ref{thm:global} to find a pair
$(\pi_1, \pi_2)$ of non-equivalent representations such that both
representations have a global $(T, \chi)$-Waldspurger model.  Let
$v \not\in S$ be a place where $T_v$ is split. Let $\pi_{1, v},
\pi_{2, v}$ be two non-equivalent supercuspidal representations of
$D_v^\times \simeq \GL_2(F_v)$ with the same central character as
$\chi_v \big\vert_{Z_v}$. Then we know that there are automorphic
representations $\pi_1, \pi_2$ with $(T, \chi)$-Waldspurger models
and such that their respective local components at $v$ are
$\pi_{1, v}, \pi_{2, v}$. Since $\pi_{1, v}, \pi_{2, v}$ are
non-equivalent, we have $\pi_1 \not\simeq \pi_2$, and we are done.
By \S 13 of \cite{P-TB} the representation $\theta(\pi_1, \pi_2)$
will have a $(T, \chi)$-Bessel model and hence cannot be zero.
Note that on the other hand since $\theta(\pi_1, \pi_2)$ is
non-zero, it is locally non-zero. The local component of
$\theta(\pi_1, \pi_2)$ is $\theta(\pi_{1, v}, \pi_{2, v})$. This
then means that $\theta(\pi_{1,v}^{JL}, \pi_{2, v}^{JL}) \ne 0$.
Now since $\pi_1^{JL}, \pi_2^{JL}$ are generic, and there is no
local obstruction, the non-vanishing of global theta lift follows
from Theorem \ref{thm:takeda}.
\end{proof}

\end{document}